\documentclass{amsart}
\usepackage{amsthm}
\usepackage{amsfonts}
\usepackage{amsmath}
\usepackage{caption}
\usepackage{graphicx}
\usepackage{comment}
\usepackage{amscd}
\usepackage{tikz-cd}

\newcommand{\KK}{\mathbb{K}}
\newcommand{\QQ}{\mathbb{Q}}
\newcommand{\ZZ}{\mathbb{Z}}

\newcommand{\Aut}{\mathrm{Aut}}

\newcommand{\OO}{\mathcal{O}}

\newcommand{\Cl}{\mathrm{Cl}}
\renewcommand{\div}{\mathrm{div}}

\theoremstyle{plain}
\newtheorem{lem}{Lemma}
\newtheorem{prop}{Proposition}
\newtheorem{theor}{Theorem}

\newtheorem{conj}{Conjecture}
\theoremstyle{definition}
\newtheorem{definition}{Definition}
\newtheorem{ex}{Example}
\theoremstyle{remark}
\newtheorem{remark}{Remark}

\begin{document}

\title[Automorphism group orbits on horospherical varieties]{Automorphism group orbits on horospherical varieties and divisor class group}
\author{Sergey Gaifullin}
\address{Moscow Center for Fundamental and Applied Mathematics, Moscow, Russia; \linebreak and \linebreak
National Research University Higher School of Economics, Faculty of Computer Science, Pokrovsky Boulevard 11, Moscow, 109028, Russia}
\email{sgayf@yandex.ru}



\thanks{The author was supported by RSF grant 20-71-00109.}

\maketitle
\begin{abstract}
In 2013 Bazhov proved a criterium for two points on a complete toric variety to lie in the same orbit of the neutral component of automorphism group. This  criterium is in terms of divisor class group. Arzhantsev-Bazhov (2013) obtained a similar criterium for affine toric varieties. We prove a necessary condition similar this criteria to the cases of affine and projective horospherical varieties.
\end{abstract}

\section {Introduction}

Let $\KK$ be an algebraically closed field of characteristic zero. Let $X$ be a projective algebraic variety. We can consider the group of its regular automorphisms $\Aut(X)$.  Then $\Aut(X)$ is a linear algebraic group. We can consider its neutral component. We investigate orbits of the natural action of this group on $X$.
 
 If a connected linear algebraic group acts of $X$, then $G$ maps to $\Aut(X)^0$. Therefore, $\Aut(X)^0$-orbits are unions of $G$-orbits. So, if we have a description of $G$-orbits, we need only to understand which pairs of $G$-orbits can be glued by an element of $\Aut(X)^0$. This approach is very useful if there are only finite number of $G$-orbits. In this paper we investigate $\Aut(X)^0$-orbits on projective horospherical varieties of complexity zero. 
 Let us recall that  a variety is called {\it horospherical}, if it admits such an action of a connected linear algebraic group $G$, that generic points has a maximal unipotent subgroup in isotropy group.  We say that a horospherical variety has complexity zero if the action $G$ on $X$ has an open orbit. The class of horospherical varieties is a natural generalization of toric varieties and a subclass in class of spherical varieties.  
 
In paper~\cite{B} all $\Aut(X)^0$-orbits of complete toric varieties are described. In~\cite{AB} a description of $\Aut(X)^0$-orbits for an affine toric variety is obtained. Note that for an affine variety the automorphism group $\Aut(X)$ can be nonalgebraic. We define the neutral component $\Aut(X)^0$ as a subgroup of $\Aut(X)$, generated by all algebraic families  of automorphisms. Here an algebraic family is  a set $\{\varphi_b\mid b\in B\}$ of automorphisms, where $B$ is an irreducible algebraic variety,  the mapping 
$\varphi\colon B\times X\rightarrow X$ is a morphism, and $\varphi_b(x)=\varphi(b,x)$.

Let $X$ be a complete toric variety. Then there are finite number of $T$-orbits on it. Closures of orbits of codimension 1 are exactly $T$-invariant prime divisors on $X$.  
Consider a point  $x\in X$. Denote by $D(x)$ the set of prime $T$-invariant divisors $D$ such that $x\notin D$. It is clear that if $x$ and $y$  lie on the same $T$-orbit, then $D(x)=D(y)$.  So, we can denote by  $D(\OO)$ the set $D(x)$ for any $x$ in a $T$-orbit $\OO$. Now let us consider a submonoid
 $\Gamma(\OO)$ in the divisor class group $\Cl(X)$ generated by classes of divisors from  $D(\OO)$. 
In~\cite{B} the following assertion is proved. 
\begin{prop}\cite[Theorem~1]{B}\label{baz}
Let $X$ be a completetoric variety. Two $T$-orbits $\OO$ and $\OO'$ lie in the same $\Aut(X)^0$-orbit if and only if
$$\Gamma(\OO)=\Gamma(\OO').$$
\end{prop}
 The same criterium holds for affine toric varieties, see~\cite{AB}.

Our goal is to investigate $Aut(X)^0$-orbits of affine and projective irreducible normal horospherical varieties of complexity zero. In the case of horospherical variety we define $D(x)$ as the set of  prime $G$-invariant divisors $D$ such that $x\notin D$. Again it is clear that $D(x)$ depends only on $G$-orbit, so we use notation $D(\OO)$. But for horospherical variety we replace $\Cl(X)$ to its modification $\Cl_G(X)$, see Definition~\ref{ddd}.
Then we can define $\Gamma_G(\OO)$ by the way similar to the case of toric varieties. We prove that condition of Proposition~\ref{baz} remains necessary, see Theorem~\ref{main}.

But this condition is not sufficient in case of horosherical varieties, see Example~\ref{se}. We state a conjecture that if points in $\OO$ and $\OO'$ have equal dimensions of tangent spaces, then  this condition is sufficient. See Conjecture~\ref{ccc}, compare to \cite[Conjecture~1]{BGSh}. 

The author is a Young Russian Mathematics award winner and would like to thank its sponsors and jury.

\section{Preliminaries}\label{tv}

\subsection{Toric varieties}

In this section we remind basic facts about toric varieties. More information about toric varieties one can find, for example, in~\cite{CLSH} and~\cite{Fu}. A toric variety is a normal irredusible variety $X$ with a fixed action of an algebraic torus $T\cong (\KK^\times)^n$ with an open orbit.  Let $N\cong \ZZ^n$ be the lattice of one-parameter subgroups of $T$, and let $M\cong \ZZ^n=N^*$ be the dual to $N$ lattice of characters of~$T$. Denote by $\chi^m\in \KK[T]$ the character, corresponding to a vector $m\in M$. By~$\langle\cdot,\cdot\rangle\colon M\times N\rightarrow \ZZ$ we denote the natural pairing between these lattices. Let us also introduce  $M_\QQ=M\otimes_\ZZ\QQ$ and $N_\QQ=N\otimes_\ZZ\QQ$ which are rational vector spaces spanned by these lattices. Each toric variety corresponds to a fan $\Delta$ of cones in $N_\QQ$. For an affine toric variety this fan consists of one cone and all its faces. Define the dual cone $\sigma^\vee\subseteq M_\QQ$ by
$$\sigma^\vee=\{m\in M\mid \langle m,n \rangle\geq 0| n\in\sigma\}.$$ then the toric variety $X(\sigma)$, corresponding to $\sigma$ has the following algebra of regular functions 
$$\KK[X]=\bigoplus_{m\in M\cap \sigma^\vee}\KK\chi^m.$$
There is a one-to-one correspondence between faces of $\sigma$ and faces of $\sigma^\vee$. A face $\tau$ of $\sigma$ corresponds to the face $\widehat{\tau}=\sigma^\vee\cap  \tau^\bot$. Also there is a one-to-one correspondence between faces of $\sigma^{\vee}$ and $T$-orbits of $X$, see the next section for description of this correspondence in more general case of horospherical varieties.  

An arbitrary toric variety can be obtained by gluing affine charts corresponding to cones of $\Delta$. A variety is complete if and only if union of all cones in $\Delta$ coincides with $M_\QQ$.

\subsection{Affine horospherical varieties}

We recall some results on horospherical varieties, all proofs can be found in \cite{PV}, see also \cite{T}.

Let $G$ be a connected linear algebraic group. 

\begin{definition}
An irreducible $G$-variety $X$ is called \emph{horospherical}, if for a generic point $x\in X$ the stabilizer of $x$ contains a maximal unipotent subgroup $U\subseteq G$.

If $X$ contains an open $G$-orbit, then $X$ is called \emph{complexity-zero} horospherical variety. In~\cite{PV} affine complexity-zero horospherical varieties are called $S$-varieties.
\end{definition}

Suppose that $X$ is an affine complexity-zero horospherical variety. It is easy to see that the unipotent radical of $G$ acts trivially on $X$. Hence we may assume that $G$ is reductive. Taking a finite covering, we may assume that $G=T\times G'$, where $T$ is an algebraic torus and $G'$ is a semisimple group.

Let $O$ be the open orbit in $X$. We have the following sequence of inclusions
$$
\KK[X]\hookrightarrow\KK[O]\hookrightarrow\KK[G].
$$

Let $B$ be a Borel subgroup of $G$ and let $M=\mathfrak{X}(B)$ be the group of characters of $B$.   
For a $\Lambda\in M$ we put
$$
S_\Lambda=\left\{ f\in\KK[G]\mid f(gb)=\Lambda(b)f(g)\ \text{for all} \ g\in G, b\in B\right\}.
$$
Then
$$
S_{\Lambda} S_{\Lambda'}=S_{\Lambda+\Lambda'}.
$$

The set $\mathfrak{X}^+(B)$ of dominant weights consists of all $\Lambda$ such that $S_\Lambda\neq\{0\}$. 
It is proved in~\cite{PV} that for an affine complexity-zero horospherical $G$-variety $X$ there is a decomposition
$$
\KK[X]=\bigoplus_{\Lambda\in P} S_\Lambda
$$
for some submonoid $P\in\mathfrak{X}^+(B)$. 

Using notations from the previous section, we denote by $\sigma^\vee$ the cone in $M_\QQ$ spanned by~$P$. The variety $X$ is normal if and only if $P$ is saturated, i.e. $P=\sigma^{\vee}\cap \ZZ(P)$, where $\ZZ(P)$ is the group generated by $P$. There is a one-to-one correspondence between faces of~$\sigma$ and $G$-orbits of $X$. More precisely, if
$O_\tau\subseteq  X$ is the $G$-orbit in $X$ corresponding to a face~$\tau$ of the cone $\sigma$, then the ideal of functions vanishing on $O_\tau$ has the form 
$$
I(O_\tau)=\bigoplus_{\Lambda\in P\setminus\tau}S_\Lambda.
$$ 
This ideal vanishes on the closure $\overline{O_\tau}$. Then 
$$O_\tau=\overline{O_\tau}\setminus  \left(\bigcup_{\gamma \prec \tau}\overline{O_\gamma} \right).$$
If $\widehat{\xi} \preceq \sigma^\vee$ is the face corresponding to $\xi \preceq \sigma$, we use both denotations: $O_\xi=O_{\widehat{\xi}}$.

\begin{remark}
The cone $\sigma^\vee$ can be not of full dimension. But if we replace $M$ by the group generated by $P$, the correspondence between faces and orbits remains.
\end{remark}

To obtain a variety $X$ explicitly one should take generators $\Lambda_1,\ldots, \Lambda_m$ of $P$ and consider the sum of irreducible $G$-representation which are contragradient to ones with highest weights  $\Lambda_1,\ldots, \Lambda_m$. In each $V(\Lambda_i)^*$ one need to find the eigenvector $v_i$. Put $v=v_1+\ldots+v_m$. Then $X\cong \overline{Gv}$. If $m=1$, then the variety $X$ is the closure of the eigenvector of an irreducible representation. Such varieties are called $HV$-varieties.

\subsection{Divisor class group}

Let $X$ be a normal variety. By a prime divisor we mean a closed subset of codimension one.  Let us consider the group of Weil divisor, i.e. a group of formal integer linear combinations of prime divisors: 
$$WDiv(X)=\{\lambda_1 D_1+\ldots+\lambda_k D_k| D_i\text{ are prime divisors}\}.$$
Elements of $WDiv(X)$ we call Weil divisors on $X$.
To each rational function $f\neq 0\in \KK(X)$ and a prime divisor $D$ one can define the order $\nu_D(f)$ of $f$ on $D$. Then we can consider the  principle divisor 
$$\div(f)=\sum_{D\text{ is a prime divisor}}\nu_D(f)D.$$ All principle divisors form the group of principle divisors $PDiv(X)\subseteq WDiv(X)$. By divisor class group we mean the quotient group $\Cl(X)=WDiv(X)/PDiv(X)$. The image of $D\in WDiv(X)$ under the canonical homomorphism $WDiv(X)\rightarrow \Cl(X)$ we call class of $D$ and denote by $[D]$.

Assume that we have a $G$-action on $X$ with an open orbit, where $G$ is a connected linear algebraic group. There are finite number of $G$-orbits of codimension one. In case of toric variety they correspond to extremal rays $\rho_1,\ldots, \rho_k$ of cones $\sigma\in\Delta$.  We denote the set of extremal rays by $\Delta(1)$. Their closures are $T$-invariant prime divisors $D_1,\ldots D_k$. Let $v_{\rho_1},\ldots, v_{\rho_k}$ be primitive integer vectors on $\rho_1,\ldots, \rho_k$ In toric case if can be proved, see~\cite{Fu}, that 
$\div(\chi^m)=\sum \langle m,v_{\rho_i}\rangle D_i$ and $\Cl(X)$ is generated by $[D_1],\ldots, [D_k]$. In case of affine horospherical variety, $G$-invariant prime divisors corresponds to some (may be not all) extremal rays $\rho_1,\ldots,\rho_s$ of~$\sigma$. Then $[D_1],\ldots, [D_s]$ generates a subgroup, which can be a proper subgroup of~$\Cl(X)$.

\subsection{Locally nilpotent derivations and Demazure roots}

More information about locally nilpotent derivations one can find, for example, in \cite{Fr}. 

Let $A$ be a commutative associative algebra over $\KK$.
A linear mapping $\partial : A \rightarrow A$ is called {\it a derivation} if it satisfies the Leibniz rule: $\partial(ab) = a\partial(b) + b\partial(a)$.
A derivation is called {\it locally nilpotent or LND} if for any $a \in A$ there is a positive integer $n$ such that $\partial^n(a) = 0$.

Exponential mapping gives a correspondence between LND and subgroups in $\Aut(A)$ isomorphic to the additive group of $\KK$. An  LND $\delta$ corresponds to the subgroup $\{\exp(t\delta)\}$.
Let $F$ be an abelian group. Consider $F$-grading:

$$A = \bigoplus_{f \in F}A_f, \ \ \ A_fA_g \subseteq A_{f+g}.$$

A derivation $\partial : A \rightarrow A$ is called {\it $F$-homogeneous of degree $f_0 \in F$} if for all $a \in A_f$ we have $\partial(a) \in A_{f+f_0}.$

Consider a finitely generated cone $\sigma$.
Let $$\mathfrak{R}_{\rho}:=\{e \in M \ \ | \ \ \langle e,v_{\rho} \rangle = -1, \langle e,v_{\rho '} \rangle \geq 0 \ \ \forall \rho ' \neq \rho \in \sigma(1)\}.$$
Then the elements of the set $\mathfrak{R} := \bigsqcup\limits_{\rho}\mathfrak{R}_{\rho}$ are called the {\it Demazure roots} of the cone~$\sigma$.

We will call ray $\rho$  the {\it distinguished ray} of the Demazure root $e$ if $e \in \mathfrak{R}_{\rho}$.

Let $\tau$ be a face of $\sigma^\vee$. Let $\rho_1,\ldots, \rho_s$ be all rays normal to $\tau$ and $v_1,\ldots, v_s$ be their primitive vectors. Suppose $e$ is a Demazure root such that $\langle e, v_{1}\rangle=-1$ and  $\langle e, v_{i}\rangle=0$ for all $2\leq i\leq s$. Then we say that $e$ is a $\tau$-{\it root}. 

For an affine toric variety $X$ each Demazure root corresponds to an $M$-homogeneous LND of $\KK[X]$ given by 
\begin{equation}\label{demaz}\partial_e(\chi^m) = \langle p_{\rho}, m \rangle \chi^{e+m}.\end{equation}

\section{$G$-invariant divisor class group}

In this section we introduce a concept of a modified divisor class group of a horospherical variety $X$. Let us subdivide all prime divisors on $X$ onto two nonintersecting subsets $A$ and $B$, where $A$ is the set of all $G$-invariant prime divisors, i.e. closures of $G$-orbits of codimension one.  Then 
$$
WDiv(X)=WDiv_A(X)\oplus WDiv_B(X),\  \text{ where } 
$$
$$
WDiv_A(X)=\left\{\sum_{D_i\in A}\lambda_iD_i\right\} \ \text{ and }\  WDiv_B(X)=\left\{\sum_{D_i\in B}\lambda_iD_i\right\}.
$$
 Let us denote by $\pi_A$ and $\pi_B$ projections on $WDiv_A(X)$ and $WDiv_B(X)$. 
 
 Consider the group 
 $$HPDiv(X)=\{\div(f)|f\in\KK(X)\text{ is } M\text{-homogeneous}\}.$$ 
 
 \begin{definition}\label{ddd}
By {\it $G$-invariant divisor class group} of $X$ we call the quotient group 
$$\Cl_G(X)=WDiv_A(X)/\pi_A(HPDiv(X)).$$
 \end{definition}

For any $D\in WDiv_A(X)$ we can consider its class $[D]_G\in \Cl_G(X)$.

\begin{remark}
In case of toric $X$ the divisor class group $\Cl(X)$ is generated by classes of $T$-invariant prime divisors.  Let us take an $M$-homogeneous $f\in \KK[X]$. 
Then $\div(f)$ is $T$-invariant.  Therefore, $\div(f)\in WDiv_A(X)$. Conversely, if $\div (f)$ is in $WDiv_A(X)$, then $f$ is $M$-homogeneous. Therefore, 
$$\Cl_G(X)=WDiv_A(X)/(PDiv(X)\cap WDiv_A(X))\cong \Cl(X).$$
\end{remark}

\begin{remark}
Let $X$ be a normal affine horospherical variety. If we have an $M$-homogeneous function $f\in\KK[X]_m$, then its order on $\overline{O_{\rho_i}}$ equals to $\langle m,v_i \rangle$. Hence, $\pi_A(\div(f))=\sum_{i=1}^s\langle m, v_i\rangle D_i$. 
If all extremal rays $\rho_1,\ldots, \rho_k$ corresponds to $G$-invariant divisors, then
 group $\Cl_G(X)$ depends only on the cone $\sigma$ and it is isomorphic to $\Cl(Y)$, where $Y$ is the toric variety corresponding to the cone $\sigma$. In arbitrary case we have an algorithm of computation of the group $\Cl_G(X)$, since all relations are given by $\pi_A(\div(f))$, $f\in \KK[X]_m$ depends only on $m$.
\end{remark}

Let us finish this section by a definition of $G$-invariant analog of $\Gamma(\OO)$.

\begin{definition}
Let $X$ be a horospherical variety and $\OO$ be a $G$-orbit on $X$. Let us denote by $\Gamma_G(\OO)$ the submonoid in $\Cl_G(X)$ generated by all classes $[D_i]_G$ of $G$-invariant prime divisors $D_i$ such that $\OO\not\subseteq D_i$.
\end{definition}

 \section{Main results}
 
The main result of this paper is the following theorem.

\begin{theor}\label{main}
Let $X$ be an affine or a projective irreducible normal horospherical variety of complexity zero. If two $G$-orbits $\OO$ and $\OO'$ lie in the same $\Aut(X)^0$-orbit, then $\Gamma_G(\OO)=\Gamma_G(\OO')$.
\end{theor}

In~\cite{BGSh} the $\Aut(X)^0$-orbits for an affine horospherical variety of complexity zero were investigated. 

\begin{prop}\cite[Corollary~4]{BGSh}
Assume $X$ is an affine horospherical variety of complexity zero. Let $H$ be the subgroup generated by $G$ and all exponents of $M$-homogeneous LNDs. Then $\Aut(X)^0$-orbits on $X$ coincide with $H$-orbits.
\end{prop} 

Moreover, let $\tau$ be a face of $\sigma$, then \cite[Theorem~1]{BGSh} asserts that a closure $\overline{O_\tau}$  of a $G$-orbit $O_\tau$ is not $\Aut(X)^0$-invariant if and only if there exists a nonzero $M$-homogeneous LND with degree equals to $\widehat{\tau}$-root.  It is easy to see that exponent of an LND $\partial$ of degree $e\in\mathfrak{R}_\rho$, which is $\widehat{\tau}$-root, takes points in $O_\tau$ to points in $O_{\zeta}$, where $\zeta=\mathrm{cone}(\tau, \rho)$. Let us say that such pairs of faces are {$\rho$}-connected. So, we obtain the following lemma.

\begin{lem}\label{ch}
Assume $X$ is an affine horospherical variety of complexity zero.Two $G$-orbits $O_\tau$ lie and $O_{\tau'}$ in the same $\Aut(X)^0$-orbit if and only if there is a chain of faces $\tau=\tau_1,\tau_2,\ldots, \tau_m=\tau'$ such that $\tau_i$ and $\tau_j$ are $\xi_i$-connected for some $\xi_i\in\sigma(1)$.
\end{lem} 

Let us prove the following lemma about affine horospherical varieties. It is similar to \cite[Lemma~3.3]{B}.

\begin{lem}\label{innv}
Assume $X$ is an affine horospherical variety of complexity zero. If two $G$-orbits $O_\tau$ and $O_{\tau'}$ are $\rho_j$-connected for some $\rho_j\in\sigma(1)$. Then $\Gamma(O_\tau)=\Gamma(O_{\tau'})$. 
\end{lem} 
\begin{proof}
We can assume that $\tau$ is bigger then $\tau'$. Then there are two cases. If $\rho_j$ do not corresponde to a divisor, then  $D(O_{\tau'})=D(O_\tau)$. So, the goal is obtained. Otherwise we have $$D(O_{\tau'})=D(O_\tau)\cup\{\overline{O_{\rho_j}}\}.$$ Let $\partial$ be an LND of degree $e\in\mathfrak{R}_{\rho}$, where $e$ is a $\tau$-root.
Take a homogeneous function $f\in\KK[X]_\alpha$ such that $\partial(f)\neq 0$. Let us consider the following rational function $F=\frac{\partial(f)}{f}$. This function is $M$-homogeneous of degree $e$. Therefore 
$$\pi_A(\div(F))=\sum_{i=1}^s\langle e,v_i\rangle \overline{O_{\rho_i}}=- \overline{O_{\rho_j}}+\sum_{i\neq j}\langle e,v_i\rangle \overline{O_{\rho_i}}.$$  Therefore, in $\Cl_G(X)$ we have
$$0=- [\overline{O_{\rho_j}}]_G+\sum_{i\neq j}\langle e,v_i\rangle [\overline{O_{\rho_i}}]_G.$$
But all $\langle e,v_i\rangle$ are nonnegative for $i\neq j$. Therefore $ [\overline{O_{\rho_j}}]_G\in\Gamma_G(O_{\tau})$. Hence, $\Gamma_G(O_{\tau'})=\Gamma_G(O_{\tau})$.
\end{proof}

\begin{proof}({\it Proof of Theorem~\ref{main}})
For affine varieties Theorem~\ref{main} follows from Lemmas \ref{ch} and  \ref{innv}.

If $X$ is projective, we can consider its affine cone $\widehat{X}$. By $G$-linearization of bundel, see~\cite{KKLV} we can lift $G$-action and $\Aut(X)^0$-action on $\widehat{X}$.  The lifted $G$-action makes $\widehat{X}$ to be horospherical. To make it to be of complexity zero we can add $\KK^\times$-action by homotheties. Each $G$-orbit $\OO$ on $X$ is the projectivization of a $G\times \KK^\times$-orbit $\widehat{\OO}$ on $\widehat{X}$. And $\Gamma_G(\OO)=\Gamma_{G\times\KK^\times}(\widehat{\OO})$. So, if $\OO$ and $\OO'$ can be glued by $\Aut(X)^0$, then $\widehat{\OO}$ and $\widehat{\OO'}$ can be glued by lifting of $\Aut(X)^0$ which is contained in $\Aut(\widehat{X})^0$. Hence, $\Gamma_{G\times\KK^\times}(\widehat{\OO})=\Gamma_{G\times\KK^\times}(\widehat{\OO'})$. Therefore, $\Gamma_{G}(\OO)=\Gamma_{G}(\OO')$.
\end{proof}

\begin{ex}\label{se}
Let us show that the necessary condition proved in Theorem~\ref{main} is not a criterium. 
Let $G=\mathrm{SL}_3$ and $P=\mathfrak{X}^+(B)$. Then the cone $\sigma^\vee$ in basis of fundamental weights is $\mathrm{cone}(e_1,e_2)$.  It is easy to compute that affine 
$$X\cong\{x_1y_1+x_2y_2+x_3y_3=0\}\subset \KK^6.$$ 
The orbit, corresponding to $\tau$ is the point $q=(0,0,0,0,0,0)$. It is easy to see, that the point $q$ is $\Aut(X)$-stable since it is the unique singular point. It can be shown that there are two $\Aut(X)^0$-orbits on $X$. But each proper face correspondes to an orbit of codimension $\geq 2$. Therefore, $\Cl_G(X)$ is trivial. Hence, $\Gamma(\OO)$ does not separate  points. 
\end{ex}

Let us state the following conjecture. Compare to \cite[Conjecture~1]{BGSh}

\begin{conj}\label{ccc}
Let $X$ be an affine or a projective irreducible normal horospherical variety of complexity zero. Then two $G$-orbits $\OO$ and $\OO'$ lie in the same $\Aut(X)^0$-orbit if and only if $\Gamma_G(\OO)=\Gamma_G(\OO')$ and dimensions of tangent spaces at points of these orbits are equal.
\end{conj}

\end{document}